\documentclass[12pt,a4paper,english,reqno]{amsart}
\usepackage[a4paper,footskip=1.5em]{geometry}
\usepackage{amsmath,amssymb,amsthm,mathtools}
\usepackage[mathscr]{euscript}
\usepackage[usenames,dvipsnames]{color}
\usepackage{adjustbox,tikz,calc,graphics,babel,standalone}
\usepackage{subcaption}
\usepackage{csquotes,enumerate,verbatim}
\usepackage[final]{microtype}
\usepackage[numbers]{natbib}
\usetikzlibrary{shapes.misc,calc,intersections,patterns,decorations.pathreplacing}
\usepackage{hyperref}
\hypersetup{colorlinks=true,linkcolor=blue,citecolor=blue}

\theoremstyle{plain}
\newtheorem*{theorem*}{Theorem}
\newtheorem{theorem}{Theorem}[section]
\newtheorem{lemma}[theorem]{Lemma}

\newtheorem{proposition}[theorem]{Proposition}
\newtheorem*{claim*}{Claim}

\newtheorem{conjecture}[theorem]{Conjecture}

\theoremstyle{remark}

\DeclareMathOperator{\Aut}{Aut}
\newcommand{\I}{\mathcal{I}}
\newcommand{\A}{\mathcal{A}}
\newcommand{\B}{\mathcal{B}}

\newcommand{\PN}{\mathcal{P}_n}
\newcommand{\N}{\mathbb{N}}

\let\emptyset\varnothing
\let\eps\varepsilon
\let\originalleft\left
\let\originalright\right
\renewcommand{\left}{\mathopen{}\mathclose\bgroup\originalleft}
\renewcommand{\right}{\aftergroup\egroup\originalright}

\begin{document}

\title{On symmetric 3-wise intersecting families}

\author{David Ellis}
\address{School of Mathematical Sciences, Queen Mary, University of London, London E1\thinspace4NS, UK.}
\email{d.ellis@qmul.ac.uk}

\author{Bhargav Narayanan}
\address{Department of Pure Mathematics and Mathematical Statistics, University of Cambridge, Wilberforce Road, Cambridge CB3\thinspace0WB, UK.}
\email{b.p.narayanan@dpmms.cam.ac.uk}

\date{2 September 2016}

\subjclass[2010]{Primary 05D05; Secondary 05E18}

\begin{abstract}
A family of sets is said to be \emph{symmetric} if its automorphism group is transitive, and \emph{$3$-wise intersecting} if any three sets in the family have nonempty intersection. Frankl conjectured in 1981 that if $\A$ is a symmetric $3$-wise intersecting family of subsets of $\{1,2,\dots,n\}$, then $|\A| = o(2^n)$. Here, we give a short proof of Frankl's conjecture using a `sharp threshold' result of Friedgut and Kalai.
\end{abstract}
\maketitle

\section{Introduction}
A family of sets is said to be \emph{intersecting} if any two sets in the family have nonempty intersection. One of the best-known theorems in extremal combinatorics is the Erd\H{o}s--Ko--Rado (EKR)
theorem~\citep{EKR}, which bounds the size of an intersecting family of sets of a fixed size.
\begin{theorem}
\label{thm:ekr}
Let $k,n \in \mathbb{N}$ with $k < n/2$. If $\A$ is an intersecting family of $k$-element subsets of $\{1,2,\dots,n\}$, then $|\A| \leq \binom{n-1}{k-1}$, with equality holding if and only if $\A$ consists of all the $k$-sets that contain some fixed element $i \in \{1,2,\dots,n\}$
\end{theorem}

Over the last fifty years, many results have been obtained which bound the sizes of families of sets, under various intersection requirements on the sets in the family. Such results are often called \emph{EKR-type results}. 

Often in EKR-type results, the extremal families are highly asymmetric; this is the case in the Erd\H{o}s--Ko--Rado theorem itself, and in the Ahlswede--Khachatrian theorem \cite{AK}, for example. It is therefore natural to ask what happens to the maximum possible size of an intersecting family when one imposes a `symmetry' requirement on the family.

To make the idea of a `symmetric' family precise, we need a few definitions. For a positive integer $n \in \N$, we denote the set $\{1, 2,\dots, n\}$ by $[n]$. We write $S_n$ for the symmetric group on $[n]$ and $\PN$ for the power-set of $[n]$. For a permutation $\sigma \in S_n$ and a set $x \subset [n]$, we write $\sigma(x)$ for the image of $x$ under $\sigma$, and if $\A \subset \PN$, we write $\sigma(\A) = \{\sigma(x):x \in \A\}$. We define the \emph{automorphism group} of a family $\A \subset \PN$ by
\[\Aut(\A) = \{\sigma \in S_n:\sigma(\A) = \A\}.\]
We say that $\A \subset \PN$ is \emph{symmetric} if $\Aut(\A)$ is a transitive subgroup of $S_n$, i.e., if for all $i, j \in [n]$, there exists a permutation $\sigma \in \Aut(\A)$ such that $\sigma(i) = j$.

For an integer $r \geq 2$, a family of sets $\A$ is said to be \emph{$r$-wise intersecting} if any $r$ of the sets in $\A$ have nonempty intersection, i.e., if $x_1 \cap x_2 \cap \dots \cap x_r \neq \emptyset$ for all $x_1, x_2, \dots, x_r \in \A$. Clearly, an $r$-wise intersecting family is also $t$-wise intersecting for all $2 \le t \le r$. Since an $r$-wise intersecting family $\A \subset \PN$ is also intersecting, it cannot contain both a set and its complement, so we clearly have $|\A| \leq 2^{n-1}$. This is best-possible, since the family $\A = \{x \subset [n]: 1 \in x\}$ is $r$-wise intersecting for any $r \geq 2$. However, this family is very far from being symmetric. It is therefore natural to ask, for each $r \ge 2$, how large a symmetric $r$-wise intersecting family of subsets of $[n]$ can be. When $r=2$ and $n$ is odd, the family $\{x \subset [n]: |x| > n/2\}$ is a symmetric intersecting family of (the maximum possible) size $2^{n-1}$. However, an old conjecture of Frankl~\cite{frankl-1} asserts that symmetric $r$-wise intersecting families must be much smaller when $r\ge 3$. More precisely, Frankl conjectured that if $\A \subset \PN$ is a symmetric $3$-wise intersecting family, then $|\A| = o(2^n)$. Our purpose in this paper is to give a short proof of Frankl's conjecture; in fact, we prove the following.

\begin{theorem}
\label{thm:main}
There exists a universal constant $c>0$ such that the following holds for all $n \in \N$. If $\A \subset \PN$ is a symmetric $3$-wise intersecting family, then $|\A| \le 2^n / n^c$.
\end{theorem}

Our proof relies on certain properties of the $p$-biased measure on $\PN$ as well as a result of Friedgut and Kalai~\cite{fk} on the thresholds of symmetric increasing families. We describe these tools and then give the proof of Theorem~\ref{thm:main} in Section~\ref{sec:proof}.

An obvious example of a symmetric $r$-wise intersecting subfamily of $\PN$ is the family $\{x \subset [n]: |x| > (r-1)n/r\}$, the size of which is an exponentially small fraction of $2^n$ for any $r \geq 3$. However, for each $r \geq 3$, it is possible to construct much larger examples. We give such a construction, and state some open problems, in Section~\ref{sec:conc}.

\section{Proof of the main result}\label{sec:proof}
Before proving Theorem~\ref{thm:main}, we briefly describe the notions and tools we will need for the proof.

For $0 \le p \le 1$, we write $\mu_p$ for the \emph{$p$-biased measure} on $\PN$, defined by
\[\mu_p(\{x\}) = p^{|x|}(1-p)^{n-|x|}\]
for all $x \subset [n]$. Note that $\mu_{1/2}$ is just the uniform measure, since $\mu_{1/2} (\A) = |\A|/2^n$ for any $\A \subset \PN$.

We say that two families $\A,\B \subset \PN$ are \emph{cross-intersecting} if $x \cap y \neq \emptyset$ for all $x\in \A$ and $y \in \B$. We need the following generalisation of the simple fact that an intersecting subfamily of $\PN$ contains at most $2^{n-1}$ sets.
\begin{lemma}
\label{lemma:cross}
If $\A,\B \subset \PN$ are cross-intersecting families, then
\[\mu_p(\A) + \mu_{1-p}(\B) \le 1\]
for any $0 \le p \le 1$.
\end{lemma}
\begin{proof}
Since $\A$ and $\B$ are cross-intersecting, it is clear that $\A \subset \PN \setminus \overline{\B}$, where $\overline{\B} = \{[n] \setminus x : x \in \B\}$. Therefore,
\[\mu_p(\A) \le \mu_p(\PN \setminus \overline{\B}) = 1 - \mu_{p}(\overline{\B}) = 1-\mu_{1-p}(\B). \qedhere\]
\end{proof}

For a family $\A \subset \PN$, we write $\I(\A) = \{x\cap y: x, y \in \A\}$ for the family of all possible intersections of pairs of sets from $\A$. We require the following easy lemma that relates the $(1/4)$-biased measure of $\I(\A)$ to the $(1/2)$-biased measure of $\A$.

\begin{lemma}
\label{lemma:int}
For any $\A \subset \PN$, if $\mu_{1/2}(\A) \ge \delta$, then $\mu_{1/4}(\I(\A)) \ge \delta^2$.
\end{lemma}
\begin{proof}
Let $F$ be the map from $\PN \times \PN$ to $\PN$ defined by $F(x,y) = x \cap y$. For $j \in \{0, 1, \dots, n\}$, write $[n]^{(j)}$ for the family of all $j$-element subsets of the set $[n]$ and note that a fixed set $z \in [n]^{(j)}$ is the image under $F$ of exactly $3^{n-j}$ ordered pairs $(x,y) \in \PN \times \PN$. Consequently, writing $N_j = |F(\A \times \A) \cap [n]^{(j)}|$, we have
\[\sum_{j=0}^{n}  3^{n-j} N_j \ge |\A|^2 \ge \delta^2 2^{2n}.\]
It follows that \begin{align*}
\mu_{1/4}(\I(\A)) &= \mu_{1/4}(F(\A \times \A)) = \sum_{j=0}^{n} \left(\frac{1}{4}\right)^{j} \left(\frac{3}{4}\right)^{n-j} N_j \\
&= 2^{-2n} \sum_{j=0}^{n} 3^{n-j} N_j \ge \delta^2. \qedhere
\end{align*}
\end{proof}

We say that a family $\A \subset \PN$ is \emph{increasing} if it is closed under taking supersets, i.e., if $x \in \A$ and $x \subset y$, then $y \in \A$. It is easy to see that if $\A \subset \PN$ is increasing, then $\mu_p(\A)$ is a monotone non-decreasing function of $p$. Our main tool is the following well-known `sharp threshold' result of Friedgut and Kalai from~\citep{fk}, proved using Russo's Lemma~\citep{russo} and the so-called `BKKKL' theorem of Bourgain, Kahn, Kalai, Katznelson and Linial~\citep{bkkkl} on the influences of Boolean functions on product spaces.
\begin{proposition}
\label{prop:fk}
There exists a universal constant $c_0>0$ such that the following holds for all $n \in \N$. Let $0 < p,\eps < 1$ and let $\A \subset \PN$ be a symmetric increasing family. If $\mu_p(\A)> \eps$, then $\mu_q(\A) > 1-\eps$, where
\[q = \min\left\{1,p + c_0\left(\frac{\log(1/2\eps)}{ \log n}\right)\right\}.\eqno\qed\]
\end{proposition}

We are now ready to prove our result.

\begin{proof}[Proof of Theorem~\ref{thm:main}.]
Let $\A \subset \PN$ be a symmetric $3$-wise intersecting family. Observe that the family $\{ y : x \subset y \text{ for some } x \in \A \}$ is also symmetric and $3$-wise intersecting. Therefore, by adding sets to $\A$ if necessary, we may assume that $\A$ is increasing. 

Let $\mu_{1/2}(\A) = \delta$ and note that $\delta \le 1/2$ since $\A$ is intersecting. We may also assume that $\delta >0$ since the result is trivial if $\A$ is empty.

Since $\mu_{1/2}(\A) = \delta > \delta^2$, we may apply Proposition~\ref{prop:fk} with $p = 1/2$ and $\eps = \delta^2$ to conclude that $\mu_{q}(\A) > 1-\delta^2$, where
\[q = \min\left\{1,\frac{1}{2} + c_0\left(\frac{\log(1/2\delta^2)}{ \log n}\right)\right\}.\]
By Lemma~\ref{lemma:int}, we also have $\mu_{1/4}(\I(\A)) \ge \delta^2$. Since $\A$ is $3$-wise intersecting, it follows that $\A$ and $\I(\A)$ are cross-intersecting. Hence, by Lemma~\ref{lemma:cross}, we have $\mu_{3/4}(\A) \le 1-\delta^2$.

Now, as $\mu_{q}(\A) > 1-\delta^2$ and $\mu_{3/4}(\A) \le 1-\delta^2$, it follows from the fact that $\A$ is increasing  that $q > 3/4$. Consequently, we have
\[c_0\left(\frac{\log(1/2\delta^2)} { \log n}\right) > \frac{1}{4}.\]
It is now easy to check that $\delta < n^{-1/(8c_0)}$, proving the theorem.
\end{proof}

\section{Conclusion}\label{sec:conc}
We suspect that Theorem~\ref{thm:main} is far from best-possible. Cameron, Frankl and Kantor~\citep{frankl-2} showed that a symmetric $4$-wise intersecting subfamily of $\PN$ has size at most $2^n \exp(-Cn^{1/3})$, where $C = (\log 2/2)^{1/3}$. We believe a similar result should also hold for symmetric $3$-wise intersecting families and conjecture the following strengthening of Theorem~\ref{thm:main}.

\begin{conjecture}\label{conj:best}
If $\A \subset \PN$ is a symmetric $3$-wise intersecting family, then 
\[\log_2 |\A| \le n - cn^{\delta},\]
where $c,\delta>0$ are universal constants. 
\end{conjecture}
This would be best-possible up to the values of $c$ and $\delta$, as evidenced by the following construction communicated to us by Oliver Riordan. Let $k$ be an odd integer and let $n=k^2$, partition $[n]$ into $k$ `blocks' $B_1, B_2, \ldots,B_k$ each of size $k$, and take $\A \subset \PN$ to be the family of all those subsets of $[n]$ that contain more than half the elements in each block and all the elements in some block; in other words, 
\[\A = \{x \subset [n]: (\forall\, i \in [k] : |x \cap B_i| > k/2) \wedge (\exists \,j\in[k] : B_j \subset x)\}.\]
It is easy to see that $\A$ is symmetric and $3$-wise intersecting, and that
\[\log_2 |\A| = {n - 2n^{1/2} + o\left(n^{1/2}\right)}.\]

It is straightforward to generalise the construction described above to show that, for any $r \geq 3$, there exists a symmetric $r$-wise intersecting family $\A \subset \PN$ with
\[\log_2 |\A| = n - (r-1)n^{(r-2)/(r-1)} + o\left(n^{(r-1)/r}\right),\]
for infinitely many $n \in \N$. Let $k$ be an odd integer and let $n = k^{r-1}$. Now, consider a $k$-ary tree $T$ of depth $r-1$, so that $T$ has $(k^r-1)/(k-1)$ nodes in total (with $k^i$ nodes at level $i$ for each $i \in \{0,1,\dots,r-1\}$). A node at level $r-1$ is called a \emph{leaf}, and the set of leaves of $T$ is denoted by $\mathcal{L}(T)$. Identify the ground-set $[n]$ with the set of leaves $\mathcal{L}(T)$ and take $\A$ to be the family of sets $x\subset \mathcal{L}(T)$ such that
\begin{enumerate}
\item $x$ contains more than half the leaf-children of each node at level $r-2$, and 
\item for each $l \in \{0,1,2,\dots,r-3\}$ and for each node $v$ at level $l$, there exists a child $w$ of $v$ such that $x$ contains all the leaf-descendants of $w$. 
\end{enumerate}
It is easy to see that $\log_2 |\mathcal{A}| = n - (r-1)n^{(r-2)/(r-1)} + o(n^{(r-1)/r})$, and that $\A$ is symmetric and $r$-wise intersecting.

Let us also mention the following elegant projective-geometric construction related to us by Sean Eberhard (that produces slightly smaller families). Let $q$ be a prime power, and let $\mathbb{P}^r(\mathbb{F}_q)$ denote the $r$-dimensional projective space over the field $\mathbb{F}_q$. Now, take $\A$ to be the family of all subsets of $\mathbb{P}^r(\mathbb{F}_q)$ that contain an $(r-1)$-dimensional projective subspace. Clearly, $\mathcal{A}$ is symmetric and $r$-wise intersecting. 
The number of $(r-1)$-dimensional projective subspaces of $\mathbb{P}^r(\mathbb{F}_q)$ is $(q^{r+1}-1)/(q-1)$, and each such subspace has
cardinality $(q^r-1)/(q-1)$. Hence, writing $n = |\mathbb{P}^r(\mathbb{F}_q)| = (q^{r+1}-1)/(q-1)$, we have
\[2^{n - (q^r-1)/(q-1)} \leq |\A| \leq \left(\frac{q^{r+1}-1}{q-1}\right)2^{n - (q^r-1)/(q-1)},\]
so $\log_2 |\A| = n - n^{(r-1)/r} + o(n^{(r-1)/r})$.

Finally, it would be very interesting to determine more precisely, for each $r \geq 3$, the asymptotic behaviour of the function 
\[f_r(n) = \max\{|\A|:\A \subset \PN \text{ such that } \A \text{ is symmetric and }r\text{-wise intersecting}\}.\]

\section*{Acknowledgements}
We would like to thank Sean Eberhard and Oliver Riordan for relating the constructions described above to us. These constructions improved the power of $n$ in our original construction (from $n^{\log r / \log (r+1)}$ to $n^{(r-1)/r}$ initially, and to $n^{(r-2)/(r-1)}$ subsequently).

\bibliographystyle{amsplain}
\bibliography{symmetric_3_families}

\end{document}